\documentclass[12pt]{article}
\usepackage{authblk}
\usepackage{etex}
\usepackage[utf8]{inputenc}
\usepackage[english]{babel}
\usepackage{amsmath,amsthm,amssymb}
\usepackage{sectsty}
\usepackage{titlesec}
\usepackage{color} 
\usepackage{amsgen}
\usepackage{amstext}
\usepackage{amsbsy}
\usepackage{amsopn}
\usepackage{amsfonts}
\usepackage{enumerate}
\usepackage{eqnarray}
\usepackage{mathrsfs}
\usepackage{url,fdsymbol,accents}
\usepackage{fullpage}
\usepackage{tikz-cd}
\usetikzlibrary{automata, arrows.meta, positioning,decorations.pathmorphing}

\newcommand{\footrecall}[1]{%
} 
\usetikzlibrary{snakes,shapes,arrows,automata}

\newcommand\numberthis{\addtocounter{equation}{1}\tag{\theequation}}

\def\U#1{U_{\neg{#1}}}
\def\T#1{T_{\rm #1}}
\def\shell#1{\accentset{\smallfrown}{#1}}
\newcommand{\N}{\mathbb{N}}

\newcommand{\F}{\mathcal{F}}

\newcommand{\mc}{\mathcal}

\theoremstyle{definition}
\newtheorem{theorem}{Theorem}[section]
\newtheorem{corollary}[theorem]{Corollary}
\newtheorem{definition}[theorem]{Definition}
\newtheorem{conjecture}[theorem]{Conjecture}
\newtheorem{proposition}[theorem]{Proposition}

\newtheorem{lemma}[theorem]{Lemma}
\newtheorem{example}[theorem]{Example}
\newtheorem{remark}[theorem]{Remark}
\begin{document}
\baselineskip=1.2\baselineskip
\title{On supratopologies, normalized families\\ and Frankl's
  conjecture}

\author{Andr\'e Carvalho\thanks{andrecruzcarvalho@gmail.com} }
\author{António Machiavelo\thanks{ajmachia@fc.up.pt}} \affil{CMUP,
  Faculdade de Ciências da Universidade do Porto

  Rua do Campo Alegre

  4169-007, Porto, Portugal}

\maketitle

\begin{abstract}
  We introduce some generalized topological concepts to deal with
  union-closed families, and show that one can reduce the proof of
  Frankl's conjecture to some families of so-called supratopological
  spaces.  We prove some results on the structure of normalized
  families, presenting a new way of reducing such a family to a
  smaller one using dual families. Applying our reduction method, we
  prove a refinement of a conjecture originally proposed by Poonen.
  Finally, we show that Frankl's Conjecture holds for the class of
  families obtained from successively applying the reduction process
  to a power set.
\end{abstract}
\section{Introduction}

Let $\F$ be a finite family of sets. In this context, by ``family of
sets'' we mean a set consisting of sets. We say that $\F$ is
\emph{union-closed} if, for any $F, G\in\F$, we have that
$F\cup G\in\F$. We define the \emph{universe} of a family $\F$ as the
union of all its member-sets, and denote it by $U(\F)$. Of course,
$\F\subseteq \mc{P}(U(\F))$, and if $\F$ is union-closed, then
$U(\F)\in \F$.

Any bijection between two finite sets, $U\to V$, induces a bijection
$\mc{P}(U)\to\mc{P}(V)$, that, of course, preserves unions, and
therefore union-closed families, as well as all the properties about
them that are pertinent in this paper. Thus, we will only be
interested in families modulo the equivalence relation induced by
bijections on the respective universes. We will then call two families
\emph{isomorphic} if such a bijection exists, and, as usual, we will
say that a family is ``such and such'' \emph{up to bijection}, meaning
up to a bijection between their universes.

For $S\subseteq U(\F)$, we define
$\F_S=\{F\in\F: F\cap S\neq \emptyset\}$ and
$\F_{\neg S}=\{F\in\F: F\cap S = \emptyset\}$.  For $a\in U(\F)$, we
denote $\F_{\{a\}}$ by $\F_a$, and similarly
$\F_{\neg\scriptscriptstyle{\{a}\}}$ by $\F_{\neg a}$.

The so called \emph{Frankl's conjecture}, also named the\emph{
  union-closed sets conjecture} has been attracting the curiosity of
many for a long time (see \cite{BS13}), mostly due to its apparently
simple statement.

\begin{conjecture}[Frankl conjecture]
  \label{origconjecture}
  If a finite family of sets $\F\neq\emptyset, \{\emptyset\}$ is
  union-closed, then there is an element of its universe that belongs
  to at least half the sets of the family, i.e.
  $$\exists\, a\in U(\F):  |\F_a|\geq \frac{|\F|}{2}.$$
\end{conjecture}

The origin of Frankl's conjecture is not completely clear. According
to \cite{BBE13} it was well known by the mid-1970s as a ``folklore
conjecture'', but it is usually attributed to Péter Frankl, as he
stated it in terms of intersection-closed set families in 1979. The
first such attribution seems to have being done by Dwight Duffus, in
1985, in \cite{Duffus}.

The problem has been studied from several viewpoints and some
interesting formulations have been obtained. For example, the
conjecture admits a lattice-theoretical version, which has been proved
firstly for modular lattices \cite{AN98} by Abe and Nakano, and later
for lower semimodular lattices by Reinhold \cite{Rei00}; it also
admits a graph-theoretical version which is trivially true for
non-bipartite graphs and it was proved to hold for the classes of
chordal bipartite graphs, subcubic bipartite graphs, bipartite
series-parallel graphs and bipartitioned circular interval graphs in
\cite{BCST15}; and there is also a very interesting, yet seemingly
unfruitful formulation, known as \textit{the Salzborn formulation},
described in \cite{Woj92}.  Compression techniques have been
attempted, yielding some partial results. They were introduced in this
context by Reimer in \cite{Rei03}, and later further explored in
\cite{Rod12} and \cite{BBE13}.  The concept of Frankl's Complete
families, or FC-families, introduced by Sarvate and Renaud
\cite{SR89}, and later formalized by Poonen \cite{Poo92}, has also
been studied by some (\cite{Mor06} and \cite{Vau02}, for example). A
more direct approach on the properties a hypothetical counterexample
of minimal size was taken by Lo Faro in \cite{LF94} and \cite{LoF94},
with some interesting results.

A very thorough survey on the topic by Bruhn and Schaudt,
\cite{BS13}, is suggested to the interested reader, as well as the
Master thesis of the first author \cite{Car16}.

Recently, a breakthrough by Gilmer \cite{Gil22} provided the first
known constant lower bound on the frequency of the most frequent
element in union-closed families. This was later improved, and the
current best bound is of 0.38234 \cite{Cam22,Yu23}.

In this paper, we will focus on the study of \emph{normalized}
union-closed families. These constitute an interesting subclass of
union-closed families, that are, in some sense, the smallest possible
separating families for a given universe. They are particularly
relevant because the above mentioned Salzborn formulation of Frankl's
conjecture only refers to normalized union-closed families, as opposed
to Frankl's conjecture that concerns all union-closed families of
sets.

We will study normalized families, proving some of their properties,
and detailing their construction. The main result of this paper is a
reduction technique that can be applied to any normalized family to
produce a smaller one.  Denoting by $\F\ominus S$ the set
$\{F\setminus S: F\in\F\}$, where $\F$ is a family of sets and
$S\subseteq U(\F)$, our result can be stated as:

\newtheorem*{ogrande}{Theorem~\ref{ogrande}}
\begin{ogrande}
  Let $\mc{N}$ be a $n$-normalized family and let $M$ be any minimal
  non-empty set of $\mc{N}$. Then the family
  $\mc{N}'=(\mc{N}\setminus \{M\})\ominus \{a_{\mc{N}}\}$ is
  $(n-1)$-normalized, for some $a_{\mc{N}}\in \mc{N}$.
\end{ogrande}

Poonen, in \cite{Poo92} made the following refinement of Frankl's
conjecture:
\newtheorem*{poonen}{Conjecture~\ref{poonen}}
\begin{poonen}
  Let $\mc{F}$ be a union-closed family of sets. Unless $\F$ is a
  power set, it contains an element that appears in strictly more than
  half of the sets.
\end{poonen}
\noindent We use the reduction introduced by Theorem~\ref{ogrande}
to prove that this statement can be weakened in the following way.
\newtheorem*{nossaguess}{Conjecture~\ref{nossaguess}}
\begin{nossaguess}
  Let $\mc{F}$ be a union-closed family such that the most frequent
  element belongs to exactly half the sets in $\F$. Then $\F$ must be
  a power set.
\end{nossaguess}
\noindent Note that this conjecture only concerns families
\emph{sharply} satisfying the conclusion of Frankl's conjecture,
apparently making no claim on the original conjecture. However, we
show that:
\newtheorem*{poonenequiv}{Theorem~\ref{poonenequiv}}
\begin{poonenequiv}
  Conjectures~\ref{poonen} and~\ref{nossaguess} are equivalent.
\end{poonenequiv}
\noindent Naturally, the non-trivial part is proving that
Conjecture~\ref{nossaguess} implies the Frankl conjecture.

Finally, using the reduction of the statement of
Theorem~\ref{ogrande}, we introduce a reduction technique for
arbitrary families, and use it to prove that Frankl's conjecture holds
for a certain class of families obtained from successively reducing a
power set, families that we call descendents of the original family.

\newtheorem*{descpower}{Theorem~\ref{descpower}}
\begin{descpower}
  If a family is a descendent of a power set, then it satisfies the
  Frankl Conjecture.
\end{descpower}

The paper is organized as follows. In Section~\ref{prelim}, we present
some preliminary notions and structural results on union-closed
families and on Frankl's conjecture, and show that the statement of
Frankl's conjecture can be generalized in a natural way: if true, then
one can guarantee the existence of a subset of size $k$ being in at
least $\frac1{2^k}$ of the sets of a union-closed family. In Section
\ref{GenTop}, we show the relevance of some generalized topological
concepts, namely supratopologies, for the study of union-closed
families, introducing some separation axioms and prove that Frankl's
conjecture can be reduced to families satisfying some of those axioms.
In Section \ref{norm}, we start by presenting the already known
relation between union-closed and normalized families in detail,
combining the ideas from \cite{Woj92} and \cite{JTPV97}. We then prove
Theorem \ref{ogrande}, establishing a natural way to reduce a
normalized family to a smaller one, which looks non-trivial when seen
in the original, non-normalized, context.  Finally, we present some
properties of normalized families obtained through this reduction and
some of its connections to Frankl's conjecture, including the
weakening of Poonen conjecture, in Section~\ref{frankl}, and we prove
that all families descending from power sets satisfy the Frankl
conjecture.
\section{Preliminaries and an equivalent formulation}
\label{prelim}

Let $\F$ be a family of sets.  As usual, a set $F\in \F$ is said to be
\emph{irreducible} in $\F$ if, for all $G,H\in \F$, we have
$F = G\cup H\implies G=F$ or $H=F$. The set of all irreducible non-empty
sets in a family $\F$ is denoted by $J(\F)$. Given a family
$\mc G\subseteq \mc P(U)$, we denote by $\langle \mc G\rangle$ the
smallest union-closed family in $\mc{P}(U)$ that contains $\mc G$,
which is the family whose elements are all possible finite unions of
elements of $\mc G$, the empty set included, as it is the union of an
empty family. When $\F = \langle\mc G\rangle$, we say that $\mc G$ is
a generating family for $\F$. Of course, $\F$ is generated by its
irreducible sets.

The family $\F$ is said to \emph{separating} if, for every two
distinct elements $a,b\in U(\F)$, there is a set $O\in\F$ such that
$|O\cap\{a,b\}|=1$, i.e.~$\F_a\neq \F_b$. A family is called
\emph{normalized} if it is a separating union-closed family such that
$\emptyset \in \F$ and $ |\F|= |U(\F)| + 1$. We will call it
\emph{$n$-normalized} to mean that $|U(\F)|=n$.  Setting
$[n]=\{1,2,\ldots,n\}$, a simple example of an $n$-normalized family
is given by the \emph{$n$-staircase family}
$\{\emptyset, [1],\ldots, [n]\}$. Another example is the family
$\{\emptyset\}\cup \{[n]\setminus\{a\}: a\in\{2,3,\ldots,n\}\}
\cup \{[n]\}.$

We write $\U{a}$ instead of $U(\F_{\neg a})$ when the family involved
is clear, which is thus the set of all elements belonging to some set
of $\F$ that does not contain $a$.  The following results show that
normalized families are the smallest possible families that are both
separating and union-closed.

\begin{lemma}
  \label{iffseparating}
  A family $\F$ is separating if and only if $\U{a}\neq \U{b}$, for all
  $a\neq b\in U(\F)$.
\end{lemma} 

\begin{proof}
  Suppose $\F$ is separating and let $a\neq b\in U(\F)$. Without loss
  of generality, there is  a set $O\in\F$ such that $a\in O$ and
  $b\not\in O$. Then, $a\in O\subseteq \U{b}$ and $a\not\in
  \U{a}$. Hence, $\U{a}\neq \U{b}$.

  Conversely, if $\F$ is not separating, there are $a\neq b\in U(\F)$
  such that $\F_a=\F_b$. Thus, $\F_{\neg a}=\F_{\neg b}$, and
  therefore $\U{a}= \U{b}$.
\end{proof}

\begin{proposition}
  \label{tamanhoseparating}
  A separating union-closed family of sets $\F$ with a universe with
  $n$ elements has at least $n$ sets. If, moreover, $\F$ is
  normalized, then there is $a\in U(\F)$ such that $\U{a}=\emptyset$,
  i.e.~$a$ belongs to every non-empty set in $\F$.
\end{proposition}

\begin{proof}
  Since $\F$ is union-closed, then $U(\F)\in \F$ and $\U{a}\in\F$, for
  all $a\in U(\F)$ such that $\F_a\neq \F$.  Also, from the fact
  that $\F$ is separating, it follows that there can only be at most
  one $a\in U(\F)$ such that $\F_a = \F$ and so, by the previous
  proposition, we have that $\F$ has at least $n$ sets, and at least
  $n+1$ in case there is no element belonging to all sets.

  When $\F$ is normalized, since we require that $\emptyset\in \F$,
  if there was no element belonging to every non-empty set in $\F$,
  then there would be at least $n+2$ elements in $\F$, by the
  argument in the last paragraph.
\end{proof}

We now present a generalization of the concept of separation in
union-closed families, which we call \emph{independence}: we can think
of it as a form of a weak separation between elements and
sets. Independent families will have a relevant role later in this
paper.

\begin{definition}
  \label{def:independent}
  A family $\F$ of sets is called \emph{independent} if, for all
  $a\in U(\F)$ and for all $S\subseteq U(\F)\setminus \{a\}$, one of
  the following conditions holds:
  \begin{itemize}
  \item there is a set $O\in\F_{\neg a}$ such that
    $O\cap S\neq\emptyset$;
  \item there is a set $O\in\F_a$ such that $O\cap S=\emptyset$.
  \end{itemize}
  We say that a family is \emph{dependent} if it is not independent.
\end{definition}

It is easy to see by the definition, that independent families are in
particular separating, by just taking $|S|=1$.  We now present a
characterization of independence that will become useful in Section
\ref{norm}.

\begin{lemma}
  \label{dependents}
  A family $\F$ of sets is dependent if and only if there exist
  $a\in U(\F)$ and $S\subseteq U(\F)\setminus\{a\}$ such that
  $\F_a=\bigcup\limits_{b\in S} \F_b$.
\end{lemma}

\begin{proof}
  Let $\F$ be a dependent family of sets. Then there exist
  $a\in U(\F)$ and $S\subseteq U(\F)\backslash\{a\}$ such that, given
  any set $O\in\F$, we have that $a\in O$ if and only if
  $O\cap S\neq\emptyset$.  So, for $O\in \F_a$, we have that
  $O\cap S\neq \emptyset$, thus we may take some element
  $b\in O\cap S$ to conclude that $O\in \F_b$. Therefore,
  $\F_a\subseteq \bigcup_{b\in S} \F_b$. Now, if $O\in \F_b$ for some
  $b\in S$, that means that $b\in O\cap S$, and so, in particular,
  $O\cap S\neq\emptyset$, from which it follows that $a\in O$. This
  shows that $\bigcup_{b\in S} \F_b\subseteq \F_a$.

  Conversely, suppose there exists $a\in U(\F)$ and
  $S\subseteq U(\F)\backslash\{a\}$ such that
  $\F_a=\bigcup_{b\in S} \F_b$. Let $O\in\F$. If $a\in O$, then there
  exists $b\in S$ such that $O\in\F_b$, from which it follows that
  $O\cap S\neq\emptyset$; if $a\not\in O$, then, for all $b\in S$,
  $O\not\in\F_b$, and so $O\cap S = \emptyset$. Hence, $\F$ is
  dependent.
\end{proof}

\begin{proposition}
  \label{indep enough}
  To prove Frankl's conjecture, it suffices to show it holds for
  independent families.
\end{proposition}

\begin{proof}
  Let $\F$ be a dependent union-closed family of sets. Then, by the
  previous proposition, there are $a\in U(\F)$ and
  $S\subseteq U(\F)\setminus\{a\}$ such that
  $\F_a=\bigcup\limits_{b\in S} \F_b$. If we consider the family
  $\F'=\F\ominus \{a\}$, we have that $|\F'|=|\F|$, because if that
  was not the case, then there would be sets $O,O\cup\{a\}\in \F$ such
  that $a\not\in O$. But then, $O\cup \{a\}\in \F_a$ implies that
  there is some $b\in S$ belonging to $O$, and so,
  $O\in \F_b\subseteq \F_a$, a contradiction. Hence, if Frankl's
  conjecture holds for $\F'$, then it holds for $\F$.  The family
  $\F'$ might not be independent, but if that is the case we continue
  this process, which will eventually stop in an independent family,
  as the cardinality of the universe decreases in each iteration.
\end{proof}

As mentioned in the introduction, some different formulations of the
union-closed sets conjecture arose in different branches of
mathematics. Out of those different formulations, one that seems very
surprising and also, so far, surprisingly unfruitful is the Salzborn
formulation that only refers to normalized families, a very restrict
subclass of union-closed families.

\begin{conjecture}[Salzborn formulation]
  \label{salzborn}
  If a finite family of sets $\F$ is normalized, then there is an
  irreducible set of size at least $\frac{|\F|}{2}$, i.e.
  $$\exists\, I\in J(\F): |I|\geq \frac{|\F|}{2}.$$
\end{conjecture}

Y. Jiang proposed, in \cite{YJ09} (the link is no longer available),
the following generalization of the Frankl conjecture:

\begin{conjecture}
  \label{general}
  Let $\F$ be a union-closed family of sets such that
  $n=|U(\F)|$. Then, for any positive integer $k\leq n$, there exists
  at least one set $S\subseteq U(\F)$ of size $k$ that is contained in
  at least $2^{-k} |\F|$ of the sets in $\F$.
\end{conjecture}

This conjecture is, in principle, not easier to prove, but it might be
useful in finding eventual counterexamples to the problem. It turns
out that it is, in fact, equivalent to Frankl's conjecture.  The
equivalence is not hard to see, but since we have not found it stated
in the literature, we include it here.

\begin{proposition}
  Let $\F$ be a union-closed family of sets with $n=|U(\F)|$. If the
  union-closed sets conjecture holds, then for any positive integer
  $k\leq n$ there are sets $S_k\subseteq U(\F)$ such that $|S_k|=k$,
  $S_k\subseteq S_{k+1}$, and such that $S_k$ is contained in at least
  $2^{-k} |\F|$ of the sets in $\F$. In particular, in that case,
  Conjecture~\ref{general} also holds.
\end{proposition}

\begin{proof}
  We proceed by induction on $k$. The case $k=1$ is trivial, since in
  this case Conjecture~\ref{general} reduces to the union-closed sets
  conjecture. Now, assume that we have some set $S_k\subseteq U(\F)$
  such that $S_k$ is contained in at least $2^{-k} |\F|$ of the sets
  in $\F$ and consider the family $\mc{G}=\{F\in\F:S_k\subseteq
  F\}$. We know that $|\mc{G}|\geq 2^{-k}|\F|$, and it is clear that
  $\mc{G}$ is union-closed. Now, take the family
  $\mc{G}\ominus S_k=\{G\setminus S_k: G\in\mc{G}\}$. This new family
  is still union-closed: if $A,B\in \mc{G}\ominus S_k$, then
  $(A\cup S_k)\cup (B\cup S_k)=(A\cup B)\cup S_k\in\mc{G}$, and
  $A\cup B \in\mc{G}\ominus S_k$. Also, clearly,
  $|\mc{G}\ominus S_k|=|\mc{G}|$. Since we assume the union-closed
  sets conjecture is valid, we know there exists an element
  $x\in U(\mc{G}\ominus S_k)\subseteq U(\F)\setminus S_k$ such that
  $x$ is in at least
  $\frac{|\mc{G}\ominus S_k|}{2}=\frac{\mc{|G|}}{2}\geq
  2^{-(k+1)}|\F|$ sets. Now just take $S_{k+1}=S_k\cup\{x\}$. We have
  $|S_{k+1}|=k+1$, and $S_{k+1}$ is contained in at least
  $2^{-(k+1)} |\F|$ of the sets in $\F$.
\end{proof}
\section{Generalized topologies and Frankl's conjecture}
\label{GenTop}

A set $X$ together with a union-closed family $\F\subseteq \mc P(X)$
was named a \emph{supratopological space} when $X\in\F$ in
\cite{Mashhour}; a \emph{generalized topological space} if
$\emptyset\in\F$ in \cite{Csaszar02}; and a \emph{strong generalized
  topological space}, e.g.~in \cite{PankajamSivaraj}, when
$\emptyset, X\in\F$.  Many topological notions can readily be extended
to this more general setting. For example, and this will be relevant
later, calling the elements of $\F$ \emph{open sets}, and their
complements \emph{closed sets}, the \emph{interior} of a subset
$A\subseteq X$ is, as usual, the biggest open set contained in $A$,
i.e.~the set $A^\circ = \bigcup\{O\in\F: O\subseteq A\}$, which is
open, as $\F$ is union-closed. Note that, trivially,
$A^{\circ\circ} = A^\circ$. Similarly, the \emph{closure} of $A$ is
the smallest closed set containing $A$, which is
$\bar{A} = \bigcap \{C \text{ closed}: A\subseteq C\} =
X\setminus(X\setminus A)^\circ$.

In this setting, the union-closed sets conjecture has the following
reformulation.
\newtheorem*{origconjecture}{Conjecture~\ref{origconjecture}}
\begin{origconjecture}[Frankl's conjecture, topological reformulation]
  If $(X,\F)$ is a finite supratopological space, then there is a
  point that belongs to at least half the open sets.
\end{origconjecture}
\noindent Note that, in this context, $\F_x$ is the set of all
neighborhoods of $x$, and one can also state the Frankl conjecture as:
there is a point whose neighborhoods consist of at least half of all
open sets.

There are several separation axioms that are quite pertinent for our
purposes, specially axioms between $T_0$ and $T_1$, that we will
shortly recall from \cite{AullThron}, to which we join the axiom that
we introduced above, in Definition~\ref{def:independent}, and we
separate as an axiom a condition in the definition of $\T{DD}$ from
\cite{AullThron}. Those axioms originated in the context of
topological spaces, but they carry over to supratopological spaces
without changes. However, there are relations among them that no
longer hold, as we will point out. In particular, some do not remain
between $T_0$ and $T_1$.

To state the axioms, it is convenient to denote the closure of $\{x\}$
by $\bar{x}$, for any given point $x\in X$ of a supratopological
space, and we will call the \emph{shadow}\footnote{In \cite{AullThron}
  this set is denoted by $[x]'$ and called the derived set of $x$.} of
$x$ to the set $\dot{x}=\bar{x}\setminus\!\{x\}$. It is easy to see
that, with our notations,
$\bar{x} = \{y\in X: \F_y\subseteq \F_x\} = X\setminus \U{x}$. Also,
following \cite{AullThron}, given two subsets, $A$ and $B$, of a
supratopological space $X$, we say that $A$ is \emph{weakly separated}
from $B$, which we denote by $A\longleftfootline B$, if there is an
open set $O$ of $X$ such that $A\subseteq O$ and
$B\cap O = \emptyset$. When dealing with a singular set, we will often
write $x$ for $\{x\}$. It is easy to see that
$\bar{x} = \{y\in X: y \not\vlongdash x\}$.

In \cite{AullThron}, the set
$\hat{x} = \{y\in X: x \not\vlongdash y\}$ is called the \emph{kernel}
of $x$, and $\shell{x} = \hat{x}\setminus\!\{x\}$ the \emph{shell} of
$x$. It is easy to see that
$\hat{x} = \{y\in X: \F_x \subseteq \F_y\} = \{y\in X: x\not\in
\U{y}\}$, the set of all elements $y$ that \emph{dominate} $x$, in the
language of \cite{BS13}.

We prefer the name ``supratopological'' because it directly suggests
that the family of open sets is union-closed, but, \textbf{from now
  on}, to simplify matters and because it really does not make much
difference, we assume that a supratopological space always contains
the empty set.

\begin{definition}[Separation Axioms]
  A supratopological space $(X,\F)$ is
  \begin{itemize}
  \item $T_0$ if, for any $x\neq y\in X$, either $x\vlongdash y$ or
    $y\vlongdash x$, which is equivalent to $\F_x\neq \F_y$.
  \item $\T{\mc I}$ if $\F$ is independent, as specified above in
    Definition~\ref{def:independent}. Of course,
    $\T{\mc I}\subseteq T_0$.
  \item $\T{UD}$ if $\dot{x}$ is a union of disjoint closed sets, for
    all $x\in X$. This is equivalent to require that, for all
    $x\in X$, $\U{x}\cup\{x\}$ is a intersection of open sets whose
    pairwise union is $X$.
  \item $\T{D}$ if $\dot{x}$ is closed, for all $x\in X$. This is
    equivalent to demand that, for all $x\in X$, there is $O\in\F$
    such that $O\cup\{x\}\in\F$, which turns out to be equivalent to
    $\U{x}\cup\{x\}\in\F$. Clearly, $\T{D}\subseteq\T{UD}$. 
  \item  $\T{iD}$ if, for all $x\neq y\in X$,
    $\dot{x}\cap\dot{y} = \emptyset$, which is equivalent to
    $\{x,y\}\cup\U{x}\cup\U{y} = X$.
  \item $\T{DD} = \T{D} \cap \T{iD}$.
  \item $\T{F}$ if, for all $x\in X$ and finite
    $S\subseteq X\setminus \{x\}$, either $x\vlongdash S$ or
    $S\vlongdash x$.
  \item $\T{FF}$ if, for any pair of finite disjoint sets
    $S_1,S_2 \subseteq X$, either $S_1\vlongdash S_2$ or
    $S_2\vlongdash S_1$. Of course, $\T{FF}\subseteq
    \T{F}$.
  \item $\T{Y}$ if, for all $x\neq y\in X$, one has
    $|\bar{x}\cap \bar{y}|\leq 1$. This is equivalent to either
    $\U{x}\cup\U{y} = X$, or $\U{x}\cup\U{y} = X\setminus\{z\}$ for
    some $z\in X$, for each pair $x\neq y$.
  \item $\T{YS}$ if, for all $x\neq y\in X$, one has
    $\bar{x}\cap \bar{y} \in \{\emptyset, \{x\}, \{y\}\}$. This is
    equivalent to
    $\U{x}\cup\U{y} \in\{X, X\setminus\{x\}, X\setminus\{y\}\}$. It is
    clear that $\T{YS}\subseteq \T{Y} \cap \T{iD}$.
  \item $T_1$ if, for all $x\neq y\in X$, one has
    $\F_x\setminus\F_y\neq \emptyset$. This is equivalent to $\{x\}$
    is closed, for all $x\in X$, or $\U{x} = X\setminus\{x\}$. It
    immediately follows that $T_1\subseteq \T{YS}$.
  \end{itemize}
\end{definition}

\noindent\textbf{Examples:}\label{examples}
\begin{itemize}
\item The \emph{indiscrete supratopology} on $X$, given by
  $\F=\{\emptyset,X\}$, is not $T_0$ when $|X| \geq 2$, but it is
  $\T{iD}$ when $|X|=2$, and thus $\T{iD}\not\subseteq T_0$.
\item Let us denote by $\binom{X}{k}$ the set of all subsets of the
  set $X$ with $k$ elements, and by $\binom{X}{\geq k}$ the set of all
  such subsets with at least $k$ elements, which is a union-closed
  family.  Note that
  $\langle \binom{X}{k}\rangle = \binom{X}{\geq k}\cup
  \{\emptyset\}$. The supratopological space given by
  $\left([4],\langle\binom{[4]}{2}\rangle\right)$ is a finite
  non-discrete $T_1$ space, something that cannot exist in the
  topological case.
\item The $n$-staircase family $\{\emptyset, [1],\ldots, [n]\}$ is
  $\T{D}$ but not $T_1$.
\item The space $X=[4]$ with the supratopology given by
  $\F=\langle\{1,2\},\{3,4\}\rangle$ is $\T{iD}$ but not $\T{UD}$.
\item For $n\geq 3$, the family
  $\{\emptyset\}\cup \{[n]\setminus\{a\}: a\in\{2,3,\ldots,n\}\} \cup
  \{[n]\}$ is $\T{UD}$ but not $\T{D}$.
\item The space $X=[5]$ with the supratopology given by
  $$\F=\{\emptyset,\{1,2,3\},\{1,4,5\},\{1,2,3,4,5\}\}$$
  is $\T{UD}$ but not $T_0$.
\item The space $X=[3]$ with the supratopology given by
  $\F=\{\emptyset, \{1,2\},\{1,3\},\{1,2,3\}\}$ is $\T{YS}$ but not
  $\T{\mc I}$.
\item The space $X=[4]$ with the supratopology given by the family
  $$\F=\{\emptyset, \{1,2\}, \{2,3\}, \{1,2,3\}, \{1,2,4\}, \{1,3,4\},
  \{1,2,3,4\}\}$$
  is $\T{DD}$ but not $\T{F}$.
\end{itemize}

\noindent\textbf{Remarks:}
\begin{itemize}
\item The notion of $T_0$ corresponds exactly what it is called
  ``separating'' in the union-closed literature, as we did in the
  previous section.  One can now give a topological proof of
  Lemma~\ref{iffseparating}: when $(X,\F)$ is $T_0$, the map $X\to\F$
  given by $x\mapsto X\setminus\bar{x}$ is injective.
\item If $X$ is not $T_0$, then there exist $x\neq y\in X$ such that
  $\bar{x} = \bar{y}$, and hence
  $\{x,y\} \subseteq \bar{x} \cap \bar{y}$. This shows that
  $\T{Y}\subseteq T_0$.
\item We will see below that $\T{D}\subseteq \T{\mc I} \subseteq T_0$.
\item $\T{DD} = \T{iD}\cap\T{D}$, by definition, and
  $\T{DD}= \T{iD}\cap\T{\mc I}$, by
  Proposition~\ref{TiDandTIimpliesTDD} below.
\item $T_1\not\subseteq \T{FF}$: the space $X=[4]$ with
  $\F = \langle\binom{[4]}{3}\rangle$ is $T_1$ but not $\T{FF}$.
\item $\T{iD} \cap T_0\subseteq \T{YS}$.
\item For topological spaces, one has the following relations (see
  \cite{AullThron}):
  $$T_1\subsetneq \T{DD} \subsetneq \T{D} \subsetneq \T{UD} \subsetneq
  T_0;$$
  $$T_1\subsetneq \T{FF} \subsetneq \T{Y} \subsetneq \T{F}
  \subsetneq \T{UD};$$
  $$T_1\subsetneq \T{DD} \subsetneq \T{YS} \subsetneq \T{Y}.$$
  Some of these do not hold for supratopological spaces, as they
  depend on the fact that the intersection of open sets is still open,
  which is not required to hold in a supratopological space. One can
  see that in this case the relations are as depicted in
  Figure~\ref{SeparationAxioms},
  \begin{figure}
    $$
    \begin{tikzcd}
      &  & T_0 \arrow[d,dash]& & \\
      &&\T{\mc I}& \T{UD}  \arrow[dl,dash]\\
      &&\T{D} \arrow[u,"(4)",dash]\arrow[dr,"(3)",dash]& \\
      \T{Y}\arrow[dr,dash]\arrow[uuurr,dash] &\T{iD}
      \arrow[d,dash]\arrow[dr,dash] &&\T{F}
      \arrow[d,dash] \arrow[ddl,"(2)",dash]\\
      &\T{YS}\arrow[dr,dash]
      &\T{DD}\arrow[uu,dash]&\T{FF}\\
      && T_1 \arrow[u,"(1)",dash]&
    \end{tikzcd}
    $$
    \caption{Hasse diagram for the separation axioms in
      supratopological spaces}
    \label{SeparationAxioms}
  \end{figure}
  where all inclusions are straightforward by the definitions
  involved, except:
  \begin{itemize}
  \item (1) holds when $\emptyset\not\in\F$, which we assumed to be
    the case in supratopological spaces.
  \item (2) and (3) follow easily from the fact that in a $\T{F}$
    space all points are either open or closed, which can be proved in
    an analogous way as Proposition~\ref{tff} below.
  \item (4) is the content of Proposition~\ref{TDimpliesTI} below.
  \end{itemize}
  We leave to the reader the verification that the examples given
  above show that there are no extra line segments in the diagram, as
  well as to verify that all inclusions are indeed strict.
\end{itemize}

It is well known that it suffices to verify Frankl's conjecture for
$T_0$ families. In fact, it is easy to see that one can suppose the
union-closed family to have, for every $a\in U(\F)$, an
\emph{$a$-problematic} set, which is a set $O\in \F$ such that
$O\cup \{a\}\in \F$. The paper by Lo Faro \cite{LF94} can be used as a
reference for this result. This means:
\begin{proposition}
  To prove the Frankl conjecture, it suffices to do it for $\T{D}$
  supratopological spaces.
\end{proposition}
We can go even further and assume $\F$ has at least 3 $a$-problematic
sets for each $a\in U(\F)$ (cf.~\cite[Corollary 3.1.11]{Car16}):

\begin{theorem}
  It suffices to prove the Frankl Conjecture for families $\F$ having
  at least three $a$-problematic sets for each $a\in U(\F)$.
\end{theorem}

\begin{proof}
  Suppose that the Frankl Conjecture does not hold, and let $\F$ be a
  minimal counterexample to the conjecture with respect to the number
  of sets and with minimal universe size among the counterexamples
  with $|\F|$ sets. Let $a\in U(\F)$ and
  $$\Pi_a=\{O\in \F_{\neg a}\mid O\cup\{a\}\in \F\}$$ be the
  subfamily of $a$-problematic sets of $\F$.  By minimality of $\F$,
  we know that Frankl's Conjecture holds for $\F'=\F\ominus\{a\}$, and
  so, $|\F'|=|\F|-|\Pi_a|<|\F|$ and $\Pi_a\neq\emptyset$.  This
  implies that $\U{a}\in\Pi_a$, as, taking $O\in\Pi_a$, we have that
  $\U{a}\cup \{a\}=\U{a}\cup (O\cup\{a\})\in \F$.  Let $b\in U(\F')$
  be an element belonging to at least $\frac{|\F'|}{2}$ sets of $\F'$,
  i.e., such that
  \begin{align}
    \label{ineq1}
    |\F_b'| \geq \frac{|\F|-|\Pi_a|}{2}.
  \end{align}
  As the element $b$ belongs to at least half the sets in $\F'$ and
  $\F_{\neg a}\subseteq \F'$ is a subfamily such that
  $|\F_{\neg a}|\geq \frac{|\F'|}{2}$, $b$ must belong to
  $ \U{a}\in\Pi_a$. In particular, $|(\Pi_a)_b|\geq 1$.

  Since $\F$ is a counterexample to Frankl's Conjecture, we have that
  \begin{align}
    \label{ineq2}
    |\F_b'| + |(\Pi_a)_b|=|\F_b|<\frac{|\F|}{2}.
  \end{align}
  Combining (\ref{ineq1}) and (\ref{ineq2}), we get that
  $$\frac{|\F|-|\Pi_a|}{2}+|(\Pi_a)_b|<\frac{|\F|}{2},$$ 
  and so $|\Pi_a|>2|(\Pi_a)_b| \geq 2.$
\end{proof}

Notice that the existence of an $a$-problematic set for each
$a\in U(\F)$ implies separation, since for $a,b\in U(\F)$,
taking $O,O\cup\{a\}\in \F$, if $b\in O$, then $O$ separates $a$
from $b$; if not then $O\cup\{a\}$ does. That is,
$\T{D}\subseteq T_0$.  The concept of independence lies somewhere
between $T_0$ and $\T{D}$, as we will now see, and which provides an
alternative way to obtain the result in Proposition~\ref{indep
  enough}.

\begin{proposition}
  \label{TDimpliesTI}
  Every $\T{D}$ union-closed family is independent,
  i.e.~$\T{D}\subseteq\T{\mc I}$.
\end{proposition}

\begin{proof}
  Suppose that $\F$ is dependent. In view of
  Lemma~\ref{dependents}, there are $a\in U(\F)$ and
  $S\subseteq U(\F)\setminus\{a\}$ such that
  $\F_a=\bigcup\limits_{b\in S} \F_b$. This means that $\F$ has no
  $a$-problematic sets, since, if there was an $a$-problematic set
  $O$, then there would be some $b\in S$ belonging to $O\cup\{a\}$,
  which means that $b\in O$ and that contradicts the fact that
  $\F_b\subseteq \F_a$. Therefore, $\F$ is not a $\T{D}$
  family.
\end{proof}

It is proved in \cite[Corollary 2]{LoF94} that, if $\F$ is a minimal
counterexample to the Frankl Conjecture, then, for all $z\in U(\F)$,
$1\leq |\hat{z}|\leq 2$, i.e. $|\shell{z}|\leq 1$. This implies that
$\F$ satisfies the axiom $\T{iD}$: given $x,y\in U(\F)$, we have that
for every element $z\in U(\F)\setminus\{x,y\}$, it cannot happen that
both $x$ and $y$ belong to $\shell{z}$, thus $z\in \U{x}\cup\U{y}$,
and so $\{x,y\}\cup\U{x}\cup\U{y}=U(\mc F)$.  This shows:
\begin{proposition}
  It suffices to prove the Frankl conjecture for $\T{iD}$ families
\end{proposition}

As noted above, the condition $\T{iD}$ is not sufficient for a family
to be $\T{DD}$. However, the following holds.

\begin{proposition}
  \label{TiDandTIimpliesTDD}
  Let $X$ be a set endowed with a supratopology $\F$. If
  $(X,\F)$ satisfies $\T{iD}$ and $\T{\mc I}$, then $(X,\F)$ is
  $\T{DD}$.
\end{proposition}

\begin{proof}
  Let $(X,\F)$ be $\T{\mc I}$ supratopological space satisfying
  $\T{iD}$, and $x\in X$. We want to prove that
  $\U{x}\cup\{x\}\in \F.$ If $\U{x}=X\setminus\{x\}$, we are done
  since $X\in \F$. If not, then $\dot{x}\neq\emptyset$. Notice
  that, if $y\in \dot{x}$, then any open set containing $y$ must also
  contain $x$, since $y\not\in \U{x}$. That is,
  $\bigcup_{y\in \dot{x}}\F_y\subseteq \F_{x}.$ Since
  $(X,\F)$ is $\T{\mc I}$, the previous inclusion has to be strict,
  and hence there must exist some $S\in \F_x$ disjoint from
  $\dot{x}=X\setminus(\U{x}\cup\{x\})$, which implies that
  $S\subseteq \U{x}\cup\{x\}$. But then
  $\U{x}\cup\{x\}=S\cup\U{x}\in \F$.
\end{proof}

We finish this section by showing that the separation axiom $\T{FF}$
is strong enough to imply Frankl's conjecture.
    
\begin{lemma}
  \label{tff}
  Let $(X,\F)$ be a supratopological space. Then $(X,\F)$ is
  $\T{FF}$ if and only if every $S\subseteq X$ is either open or
  closed.
\end{lemma}

\begin{proof} Assume that $(X,\F)$ is $\T{FF}$ and let
  $S\subseteq X$. Then, there must be $O\in F$ such that
  $S\subseteq O$ and $(X\setminus S)\cap O=\emptyset$, and so $S=O$,
  in which case $S$ is open; or $X\setminus S\subseteq O$ and
  $S\cap O=\emptyset$, and so $O=X\setminus S$, in which case $S$ is
  closed.
      
  Now, assume that every $S\subseteq X$ is either open or closed, and
  let $S_1,S_2 \subseteq X$ be disjoint. If $S_1$ is open, then taking
  $O=S_1$ in the definition of $\T{FF}$, we get that $S_1\subseteq O$
  and $O\cap S_2=\emptyset$. If, on the other hand, $S_1$ is closed,
  then taking $O=X\setminus S_1$, we obtain that $S_1\cap O=\emptyset$
  and $S_2\subseteq O$, so $(X,\F)$ is $\T{FF}$.
\end{proof}
    
\begin{proposition}
  Let $\F$ be a finite $\T{FF}$ union-closed family of sets. Then the
  Frankl Conjecture holds for $\F$.
\end{proposition}

\begin{proof} It follows from Lemma \ref{tff} that, for a $\T{FF}$
  union-closed family $\F\subseteq \mc P([n])$, we have that
  $|\F|\geq 2^{n-1}$, and so Frankl's Conjecture holds by
  \cite{karpas2017resultsunionclosedfamilies}.
\end{proof}
\section{Dual and normalized families}
\label{norm}

The notion of \emph{dual family} was introduced in \cite{JTPV97}, and
a similar notion was used in \cite{Woj92} to prove the equivalence
between the usual formulation of Frankl's conjecture and the Salzborn
formulation. The main difference between those two notions is that, in
\cite{JTPV97}, the sets of the notion presented in \cite{Woj92} are
replaced by their indexes, and the empty set is not included. Also,
the definition in \cite{JTPV97} uses a minimal generating set, while
the definition in \cite{Woj92} uses any generating set.  We use the
notion introduced in \cite{JTPV97}, with some small variations. We
will describe the construction of the dual of a given family,
illustrate it with examples, and give some of its basic properties. We
then give some structural results on normalized families, and
highlight their relation with Frankl's conjecture.

Consider an indexed subfamily $\mc{H}=\{H_1,\ldots,H_s\}$ of distinct
non-empty sets in $\mc{P}([m])$ with $U(\mc H) = [m]$.  For each
$j\in [m]$, set
$$\mc H^{\iota_j}=\{i\in[s]: j\in H_i\} \in \mc P([s]), $$
i.e., $\mc H^{\iota_j}$ is the set of indices of sets in $\mc H$ to
which $j$ belongs, and put
$${\mc H}^\iota=\{\mc H^{\iota_1},\ldots,\mc H^{\iota_m}\}.$$
Moreover, for any $A\subseteq \mc [m]$, set
$$\mc H^{\iota_A} = \{i\in[s]: A \cap H_i\neq\emptyset\} = \bigcup_{j\in A}
\mc H^{\iota_j}.$$
Note that
$\mc H^{\iota_{A\cup B}} = \mc H^{\iota_A} \cup \mc
H^{\iota_B}$. Also, note that
\begin{equation}
  \label{eq:iotaid}
  \mc H^{\iota\iota} = \mc H, \text{ since } j\in \mc H^{\iota\iota_i}
  \iff i \in \mc H^{\iota_j} \iff j\in  H_i.
\end{equation}

The choice of indices for the sets in $\mc H$ is irrelevant, as a
different choice just induces a permutation of elements on the subsets
of $\mc{P}([m])$, and one simply obtains an isomorphic family.

Now, given any subset $\mc L$ of $\mc P([m])$, we define $\mc{L}^*$ as
the union-closed family generated by $\mc H^\iota$, where
$\mc H =\mc{L}\setminus\{\emptyset\}$ is indexed in some way. From
what was noted above, $\mc{L}^*=\{\mc H^{\iota_A}: A\subseteq
[m]\}$. The special cases $\mc L=\F$ and $\mc L=J(\F)$ will be
particularly relevant. The family $\mc{F}^*$ is called the \emph{dual
  family of $\F$}.

\begin{remark}
  \label{FiotaF_A}
  Note that $\F^{\iota_A}$ is the set consisting of the indices of the
  sets in $\F_A$, and thus, in particular, $|\F^{\iota_A}| = |\F_A|$.
\end{remark}

It is clear from the definitions that
$|U(\mc L^*)|=|\mc L|-\varepsilon_{\mc L}$, , where
$$\varepsilon_{\mc L}=
\begin{cases}
  0, & \text{if } \emptyset\not\in {\mc L}, \\
  1, & \text{otherwise}.
\end{cases}
$$
and that $\mc L^*$ is separating.

The next proposition gives a procedure to build examples of normalized
families, and, in fact, as we will see below, this procedure yields
all normalized families.  The result is equivalent to one contained in
\cite[Lemma 2.2]{Woj92}, but the context is a bit different, and so
we provide a complete proof which underlies its topological nature.

\begin{proposition}
  \label{dualnormalized}
  Let $\F$ be a union-closed family, and $\mc G$ be a generating
  subfamily. Then $|\mc{G}^*|=|\F|+1-\varepsilon_\F$.  In
  particular, we have that $\F^*$ is a normalized family.
\end{proposition}

\begin{proof}
  Set $U=U(\F) = [m]$, and
  $\mc{G}\setminus\{\emptyset\} = \{H_1,\ldots, H_s\}$. As pointed out
  above, $\mc G^* = \{\mc H^{\iota_A}: A\subseteq [m]\}$, and
  therefore
  $\mc G^* = \{\mc H^{\iota_{U\setminus A}}: A\subseteq [m]\}$.  Now,
  $\mc H^{\iota_{U\setminus A}} = \mc H^{\iota_{U\setminus B}}$ is
  equivalent to saying that
  $$\forall i\in[s]\quad (U\setminus A)\cap H_i\neq \emptyset \iff
  (U\setminus B)\cap H_i\neq \emptyset,$$
  which, of course, is the same as
  $$\forall i\in[s]\quad (U\setminus A)\cap H_i= \emptyset \iff
  (U\setminus B)\cap H_i = \emptyset,$$
  or
  $$\forall i\in[s]\quad H_i\subseteq A \iff H_i \subseteq B.$$
  But this is equivalent to $A^\circ = B^\circ$, since $\mc G$ is a
  generating family for $\F$. It follows from this that
  $\mc G^* = \{\mc H^{\iota_{U\setminus A^\circ}}: A\subseteq [m]\} =
  \{\mc H^{\iota_{U\setminus O}}: O\in \F\cup\{\emptyset\}\}$, and
  that the elements of this last set are distinct, which proves first
  the claim. The second one is now very easy to establish.
\end{proof}

\begin{example}
  \label{P3m1}
  Suppose we want to construct a $6$-normalized family. To do so, we
  need a union-closed family with $7$ sets (or only $6$, if the the
  empty set is excluded). Take, for example:
  $$\mc{F}=\mc{P}([3])\backslash\{\{1\}\} =
  \{\emptyset,\{2\},\{3\},\{1,2\},\{1,3\},\{2,3\},\{1,2,3\}\}.$$
  One has $\F^{\iota_1}=\{3,4,6\}$, $\F^{\iota_2}=\{1,3,5,6\}$
  and $\F^{\iota_3}=\{2,4,5,6\}$. Now, we simply build
  \begin{align*}
    \mc{F}^*&=\langle\{3,4,6\}, \{1,3,5,6\}, \{2,4,5,6\}\rangle\\ 
            &=\{\emptyset, \{3,4,6\}, \{1,3,5,6\}, \{2,4,5,6\},
              \{1,3,4,5,6\},\{2,3,4,5,6\},  \{1,2,3,4,5,6\}\}, 
  \end{align*} 
  which is $6$-normalized.
\end{example}

The following technical lemma gives us two properties that will be
useful later on.

\begin{lemma}
  \label{techlemma}
  Let $\F$ be a union-closed family and $\mc G$ be a generating
  subfamily. We have the following:
  \begin{enumerate}
  \item if $\F$ is independent, then $J(\mc G^*)=\mc G^\iota$,
    $|J(\mc G^*)|=|U(\F)|$ and $J(\F^*)^* = \F$.
  \item if $\F$ is normalized, then $\mc G^\iota$ is
    union-closed.
  \end{enumerate}
\end{lemma}

\begin{proof}
  Let $\F$ be an independent union-closed family and put
  $U(\F)=[m]$. Clearly, in $\mc G^*$ only the sets in $\mc G^\iota$
  may be irreducible. Let
  $\mc G\setminus\{\emptyset\} =\{H_1,\ldots,H_s\}$. If
  $\mc H^{\iota_c}= \mc H^{\iota_a}\cup \mc H^{\iota_b}$, for some
  $a,b,c\in [m]$, then, since $\mc G$ is generating, it would follow
  that $\F_c=\F_a\cup \F_b$, which contradicts independence
  by Proposition~\ref{dependents}. This shows that
  $J(\mc{G}^*) = \mc G^\iota$.

  Also, since $\F$ is independent (in particular, separating), the
  sets $\mc H^{\iota_j}$ are all distinct, and so we have that
  $|J(\mc{G^*})|=|\mc G^\iota|=m$. Now, from what was proven in the
  last paragraph, we know that
  $J(\F^*)=(\F\setminus\{\emptyset\})^\iota$ and so
  $J(\F^*)^* = ((\F\setminus\{\emptyset\})^\iota)^* = \langle (\mc
  F\setminus\{\emptyset\})^{\iota\iota}\rangle= \F$.  This
  completes the proof of the first claim.
  
  Finally, let $\F$ be an $n$-normalized family. Then, by the
  previous lemma, $|\mc G^*|=n+1=|U(\F)|+1=|\mc G^\iota|+1$, since
  $\F$ is separating. This means that
  $\mc G^* = \mc G^\iota\cup \{\emptyset\}$, and thus $\mc G^\iota$ is
  union-closed.
\end{proof}

We can use this lemma to show that there is, up to bijection, only one
independent normalized family, while by definition all normalized
families are separating. This implies that the concept of independence
is stronger than the concept of separation, i.e.
$T_0\not\subset \T{\mc I}$.

\begin{proposition}
  The only independent $n$-normalized family is the \emph{staircase}
  family $\mc{N}=\{\emptyset,[1],\ldots,[n]\}$, up to bijection.
\end{proposition}

\begin{proof}
  It is easy to see that the staircase family is independent for every
  $n\in\mathbb{N}.$ Now let $\mc{N}$ be an independent $n$-normalized
  family of sets. By the previous lemma, it follows that $J(\mc{N}^*)$
  is union-closed and all its sets are, of course, irreducible. Let
  $X,Y\in J(\mc{N}^*)$. Then, $X\cup Y\in J(\mc{N}^*)$ is an
  irreducible set, and so either $Y\subseteq X$ or $X\subseteq
  Y$. Therefore $J(\mc{N}^*)$ is a chain. Now, the previous lemma also
  implies that $\mc{N}=J(\mc{N}^*)^*$. But it is easily seen that the
  dual of a chain is still a chain, and that there is only one
  separating chain of sets with a given universe, up to bijection.
\end{proof}

The next two propositions show that any $n$-normalized family can be
obtained as the dual of an independent family.

\begin{proposition}
  \label{tudo+}
  Let $\mc{N}$ be a normalized family. Then $\mc{N}=J(\mc{N})^{**}$.
\end{proposition}

\begin{proof} 
  Let $\mc{N}$ be an $n$-normalized family and set
  $\mc H =J(\mc{N}) = \{H_1,\ldots,H_s\}$.  We have
  $J(\mc{N})^*=\left\langle\mc H^{\iota_1},\ldots, \mc
    H^{\iota_n}\right\rangle.$ From Proposition~\ref{dualnormalized} it
  follows that $|J(\mc{N})^*|=|\mc{N}|=n+1$, since here $\mc{N}$
  contains the empty set. Hence
  $J(\mc{N})^*=\{\emptyset,\mc H^{\iota_1},\ldots,\mc H^{\iota_n}\}$.
  But then
  $$J(\mc{N})^{**} = \langle \mc H^{\iota\iota_1}, \ldots,
  \mc H^{\iota\iota_n} \rangle = \langle H_1, \ldots,  H_n
  \rangle =\mc{N},$$
  by \eqref{eq:iotaid}. This proves the claim.
\end{proof}

\begin{proposition}
  \label{tudo*}
  If $\mc{N}$ be a normalized family of sets, then $J(\mc{N})^*$ is
  independent. It follows that any normalized family is the dual of an
  independent family.
\end{proposition}

\begin{proof}
  In view of the previous Proposition, is enough to show that
  $\mc L = J(\mc{N})^*$ is independent. Assume this is false, so that,
  from Lemma~\ref{dependents}, and using the notations in the previous
  proof, there would exist $a\in [s]$ and
  $S\subseteq [s]\setminus\{a\}$ such that
  \begin{equation}
    \label{eq:L}
    \mc L_a = \bigcup_{b\in S} \mc L_b.
  \end{equation}
  But
  $\mc H^{\iota_j}\in \mc L_a \iff a\in \mc H^{\iota_j} \iff j\in
  H_a$, and so one sees that \eqref{eq:L} is equivalent to
  $H_a = \bigcup\limits_{b\in S} H_b$, which contradicts the fact that
  the $H_i$ are irreducible.
\end{proof}

\begin{example}
  Let $\mc{N}$ be the following $7$-normalized family:
  $$
  \mc{N}=\{\emptyset, \{1,4,6,7\},
  \{2,5,6,7\},\{3,4,5,6\},[7]\setminus\{3\},[7]\setminus\{2\},
  [7]\setminus\{1\},[7]\}.
  $$
  We have that $J(\mc{N})=\{\{1,4,6,7\},
  \{2,5,6,7\},\{3,4,5,6\}\}$. Let us compute $J(\mc{N})^{**}$ and see
  that it coincides with $\mc{N}$, as Proposition~\ref{tudo+}
  claims. To start with,
  \begin{align*}
    J(\mc{N})^*&=\langle
                 \{\{1\},\{2\},\{3\},\{1,3\},\{2,3\},\{1,2,3\},\{1,2\}\}
                 \rangle\\ 
               &=\{\emptyset,\{1\},\{2\},\{3\},\{1,3\},\{2,3\},\{1,2,3\},
                 \{1,2\}\}. 
  \end{align*}
  It follows that
  $J(\mc{N})^{**}=\langle\emptyset,\{1,4,6,7\},\{2,5,6,7\},
  \{3,4,5,6\}\rangle=\mc{N}$.
\end{example}

It is now very easy to show that the original form of the union-closed
sets conjecture and its Salzborn formulation are equivalent. We state
this result in a from that explicitly exhibits the families involved
in the equivalence.

\begin{theorem}[Salzborn Formulation, \cite{Woj92}]
  \label{franklsalz}
  If $\mc{N}$ is a normalized family and the independent family
  $J(\mc{N})^*$ satisfies the union-closed sets conjecture, then
  $\mc{N}$ satisfies the Salzborn formulation of the conjecture. If
  $\F$ is an independent union-closed family such that $\F^*$
  satisfies the Salzborn formulation of the conjecture, then
  $\F=J(\F^*)^*$ satisfies the union-closed sets conjecture.
\end{theorem}

\begin{proof}
  Let $\mc{N}$ be an $n$-normalized family, and
  $J(\mc{N}) = \{I_1,\ldots,I_s\}$. We saw in the proof of
  Proposition~\ref{tudo+} that
  $J(\mc{N})^* = J(\mc{N})^\iota\cup\{\emptyset\}$. The hypothesis
  that this set satisfies the union-closed sets conjecture entails that
  there is an element $a\in U(J(\mc{N})^*)$ in at least half the sets
  of $J(\mc{N})^*$. Now, by the definition of $J(\mc{N})^\iota$, we
  have that $a\in J(\mc{N})^{\iota_j}\iff j\in I_a$. Therefore, if we
  have $a$ in at least $\frac{n}{2}$ sets of $J(\mc{N})^*$, we have
  $|I_a|\geq \frac{n}{2}$, and so $\mc{N}$ satisfies the Salzborn
  formulation of the union-closed sets conjecture.

  To prove the second statement, let $\F$ be an independent
  union-closed family. We know by Lemma~\ref{techlemma} that
  $\F=J(\F^*)^*$. By Proposition~\ref{dualnormalized}, $\F^*$ is
  normalized, and so, by Lemma~\ref{techlemma}, $J(\F^*)^\iota$ is
  union-closed. Therefore, every set in $\F=J(\F^*)^*$ is of the form
  $J(\F^*)^{\iota_j}$. Now, if $\F^*$ satisfies the Salzborn
  condition, then, there exists $I\in J(\F^*)$ with
  $|I|\geq \frac12 |\F^*|$, and thus $|I|\geq \frac12 |\F|$, by
  Proposition~\ref{dualnormalized}. Let $i$ be the index of $I$ in
  $J(\F^*)$ used when one constructs $J(\F^*)^* = \F$. Then, since
  $k\in I \iff i\in J(\F^*)^{\iota_k}$, it follows that $i$ belongs to
  at least half of the sets in $\F$.
\end{proof}

We saw in Proposition~\ref{tamanhoseparating} that, in a normalized
family, there is an element in all of its non-empty sets, which must
be unique since the family is, by definition, separating.

\begin{definition}
  Given a normalized family $\mc{N}$, we will denote by $a_{\mc{N}}$
  the unique element belonging to all of its non-empty sets.
\end{definition}

We are now ready to present one of the main results of this paper,
which introduces a reduction process for normalized families.

\begin{theorem}
  \label{ogrande}
  Let $\mc{N}$ be a $n$-normalized family and let $M$ be any minimal
  non-empty set of $\mc{N}$. Then the family
  $\mc{N}'=(\mc{N}\setminus \{M\})\ominus \{a_{\mc{N}}\}$ is
  $(n-1)$-normalized.
\end{theorem}

\begin{proof}
  Note that if $\{a_{\mc{N}}\}\in \mc{N}$, then $M=\{a_{\mc{N}}\}$,
  and $\mc{N}'=\mc{N}\ominus \{a_{\mc{N}}\}$.  We claim that the
  family $\mc{N}'$ is $(n-1)$-normalized. It is clear that $\mc{N}'$
  is union closed, $\emptyset\in\mc{N}'$, and
  $|\mc{N}'|=|\mc{N}|-1$. It remains to show that
  $|U(\mc{N}')|=|U(\mc{N})|-1$ and that $\mc{N}'$ is separating.

  In case $\{a_{\mc{N}}\}\in \mc{N}$, it is clear that
  $|U(\mc{N}')|=|U(\mc{N})|-1$. Also, for $x,y\in U(\mc{N}')$, if
  $N\in \mc{N}$ is such that $x\in N$, $y\not\in N$, then for
  $N' = N\setminus\{a_{\mc{N}}\}\in \mc{N}'$ it is still true that
  $x\in N'$, $y\not\in N'$.

  Let us now deal with the case $\{a_{\mc{N}}\}\not\in \mc{N}$, which
  implies $|M|\geq 2$.  Clearly $|U(\mc{N}')|\leq n-1$. Now,
  $|U(\mc{N}')|<n-1$ would imply the existence of an element in $M$
  not in any other set of $\mc{N}\setminus \{M\}$, thus $M=U(\mc{N})$,
  and $\mc{N} =\{\emptyset, M\}$, but then $|M|=1$, a
  contradiction. Therefore, $|U(\mc{N}')|=n-1$.

  Finally, let $x,y\in U(\mc{N}') =
  U(\mc{N})\setminus\{a_{\mc{N}}\}$. Since $\mc{N}$ is separating,
  there exists $N\in\mc{N}$ such that $|N\cap\{x,y\}|=1$. It is, of
  course, still true that $N' = N\setminus\{a_{\mc{N}}\}$ separates
  $x$ from $y$. If $N'\neq M\setminus \{a_{\mc{N}}\}$, we are done. It
  remains to consider the case where $M$ is the only set that
  separates $x$ from $y$ in $\mc{N}$. In this case, there cannot be
  any set $\emptyset\neq L\in \mc{N}$ such that
  $\{x,y\}\not\subseteq L$, otherwise $L\cup M$ would also separate
  $x$ and $y$, and by the minimality of $M$, $M\cup L\neq M$. So,
  except for $M$, all other sets of $\mc{N}$ must contain
  $\{x,y\}$. Without loss of generality, we may assume that $x\in
  M$. It follows that $x$ belongs to all sets of $\mc{N}$,
  contradicting the uniqueness of the element $a_{\mc{N}}$.
\end{proof}

\begin{corollary}
  For every normalized family $\mc{N}$ there are distinct elements
  $a_i\in U(\mc{N})$ with frequency at least $i$, for all
  $1\leq i \leq n$.
\end{corollary}

\begin{proof}
  Follows directly from the Proposition~\ref{tamanhoseparating} by
  applying Theorem~\ref{ogrande} repeatedly.
\end{proof}

Notice that in the reduction from $\mc{N}$ to $\mc{N}'$ we remove a
minimal set. To preserve closure under union we can only remove
irreducible elements of the family. While it may be possible to weaken
the condition of minimality, it can not be replaced with
irreducibility, in general, as the next example shows.

\begin{example}
  Take the normalized family
  $$\mc{N}=\{\emptyset,
  \{5,7\},\{3,6,7\},\{3,5,6,7\},\{2,4,5,6,7\},[7]\setminus\{2\},
  [7]\setminus\{1\},[7]\}.$$ Then using the irreducible set
  $M= [7]\setminus\{2\}$ one gets
  $$\mc{N}'=\{\emptyset,
  \{5\},\{3,6\},\{3,5,6\},\{2,4,5,6\}, [6]\setminus\{1\},[6]\}$$ which
  is not separating, as no set separates $2$ from $4$.
\end{example}

Given an $n$-normalized family $\mc{N}$, we can decompose it as
$\mc{N}=\mc S_{\lambda_0}\cup \mc S_{\lambda_1}\cup \ldots \cup
\mc S_{\lambda_s}$, with $\lambda_0<\lambda_1<\cdots<\lambda_s$, where
$\mc S_\ell$ is the set of all subsets of $\mc{N}$ with $\ell$
elements. We set $k_i=|\mc S_{\lambda_i}|$.  Since $\emptyset\in \mc{N}$,
one has that $\lambda_0 = 0$ and $k_0=1$. Note that $|\mc S_1|\leq 1$. We
have the following corollary by applying
Proposition~\ref{tamanhoseparating} to all possible reductions.

\begin{corollary}
  Let $\mc{N}$ be an $n$-normalized family.  Then, using the notations
  such defined, there are $k_i$ elements with frequency at least
  $n-\sum\limits_{j=0}^{i-1}k_j$, for all $0\leq i\leq s$.
\end{corollary}

\begin{proof}
  Let $M_1$ and $M_2$ be two minimal sets of $\mc{N}$, and set
  $\mc{N}_i'=(\mc{N}\setminus \{M_i\})\ominus \{a_{\mc{N}}\}$, for
  $i=1,2$. If $a_{\mc{N}_1'}= a_{\mc{N}_2'}$, then this element, which
  is distinct from $a_{\mc{N}}$, would be in all sets of $\mc{N}$,
  contradicting the uniqueness of $a_{\mc{N}}$. The claim follows
  easily from this, by reducing $\mc{N}$ by successively removing sets
  of minimal length.
\end{proof}

Using the reduction process introduced in Theorem~\ref{ogrande} for
normalized families, one can introduce a reduction process for
arbitrary union-closed families as follows.  Given such a family $\F$,
we will call $\F_\downarrow=J((\F^*)')^*$ a \emph{child} of $\F$. It
depends on the minimal set chosen to be removed from $\F^*$ to form
$(\F^*)'$, and different choices may lead to non-isomorphic families,
but we do not include it in the notation, which would become a bit too
heavy. However, please keep in mind that there may several distinct
children of a family, and $\F_\downarrow$ just denotes one of
them. Naturally, $\F$ is called a parent of $\F_\downarrow$.  Children
of the same parent are referred as \emph{siblings}, and a family
obtained by successive reductions from a given family of sets is
called a \emph{descendent} of that family.

Note that it follows from the results seen above that
$|\F_\downarrow| = |\F|-1$, when $\F$ contains the empty set, which
will assume from now on.  As usual, we define,
$\F_{\downarrow\downarrow} = (\F_\downarrow)_\downarrow$, etc., and
$\F_{\downarrow k} = (\F_{\downarrow (k-1)})_\downarrow$, which again
may depend on various choices of minimal sets, not conveyed by the
notation. One has $|\F_{\downarrow k}| = |\F|-k$. The families
$\F_{\downarrow k}$, for $k\in\N$, are the \emph{descendents} of the
family $\F$. It follows from Proposition~\ref{tudo*} that, for any
family, all its descendents are independent.

The relation between the family $\F$ and a family $\F_\downarrow$
seems rather mysterious, in the sense that they may have different
universes, and even when the universe is the same the relation between
them does not seem to be evident.

\begin{example}
  Consider again the union-closed family $\F$ of
  Example~\ref{P3m1},
  $$\mc{F}=
  \{\emptyset,\{2\},\{3\},\{1,2\},\{1,3\},\{2,3\},\{1,2,3\}\}.$$
  There is only one minimal set in $\F^*$, namely
  $\{3,4,6\}$. It yields:
  $$(\F^*)'=\{\emptyset,\{1,3,5\},\{2,4,5\},\{1,3,4,5\},\{2,3,4,5\}
  \{1,2,3,4,5\}\}.$$
  Then $J((\F^*)')=\{\{1,3,5\},\{2,4,5\},\{1,3,4,5\},\{2,3,4,5\}\}$,
  and thus
  $$\F_\downarrow=J((\F^*)')^*=\{\emptyset,\{1,3\},\{2,4\},\{1,3,4\},
  \{2,3,4\},\{1,2,3,4\}\},$$
  which looks quite different from the original family.
\end{example}

We can picture the two parallel reduction processes in the following
commutative diagram (the left dashed up arrow commutes with the others
when $\F$ is an independent family):
\begin{equation}
  \label{eq:diagram}
  \begin{tikzcd}
    & J((\F^*)')^* \arrow[d,equal] &&& \\
    \F \arrow[r,"\square_\downarrow"] \arrow[d,"\square^*"] &
    \F_{\downarrow }\arrow[d,shift left=0.5ex]\arrow[r] &
    \F_{\downarrow \downarrow} \arrow[r,dashed]\arrow[d,shift
    left=0.5ex] & \F_{\downarrow k}
    \arrow[d,shift left=0.5ex]\arrow[r,dashed] &\hphantom{F}\\
    \F^* \arrow[u,shift left=1ex,dashed]\arrow[r,"\square'"] &
    (\F^*)' \arrow[u,shift left=0.5ex,"\square^*\circ J"]\arrow[r] &
    (\F^*)''\arrow[r,dashed] \arrow[u,shift left=0.5ex]&
    (\F^*)^{(k)}\arrow[r,dashed]
    \arrow[u,shift left=0.5ex]&\hphantom{F}\\
    & (\F_\downarrow)^* \arrow[u,equal] &&&
  \end{tikzcd}
\end{equation}
where the down arrows in the middle are given by the ``taking the
dual'' operator, $\square^*$; the right arrows on the second row are
given by the reduction operator $\square'$; the right arrows on the
first row by the reduction operator
$\square_\downarrow = \square^*\circ J\circ\square'\circ\square^*$;
and the up arrows in the middle by the $\square^* \circ J$
operator. Note that, by Theorem~\ref{ogrande} and
Proposition~\ref{tudo+}, we have that $(\F_{\downarrow k})^* = (\F^*)^{(k)}$
(i.e.~the upper arrow followed by the down arrow is the identity
operator, $Id$, or $\square^*\circ\square^*\circ J = Id$, for
normalized families). From Lemma~\ref{techlemma} it follows that, for
independent families, $\square^*\circ J \circ\square^*= Id$, i.e.~the
down arrow followed by the up arrow is the identity, for independent
families.
 
\begin{remark}
  If $\F_{\downarrow k}$ satisfies the Frankl conjecture and has an
  odd number of sets, then $\F_{\downarrow (k+1)}$ also satisfies the
  conjecture. If $\F_{\downarrow k}$ has an even number of sets and
  \emph{strictly} satisfies the Frankl conjecture, meaning that there
  is an element in strictly more than half the sets, then
  $\F_{\downarrow (k+1)}$ also {strictly} satisfies the conjecture.
\end{remark}

\begin{remark}  
  Given a normalized family $\mc{N}$, the irreducibles of
  $\mc{N}' = (\mc{N}\setminus\{M\})\ominus\{a_{\mc{N}}\}$ that do not
  belong to set $J(\mc{N})\ominus\{a_{\mc{N}}\}$ have cardinality
  bigger that $|M|$.
\end{remark}
\section{Refinement of a conjecture of Poonen and descendents of power
  sets} 
\label{frankl}

It is easy to see that, for any given $n$-normalized family $\mc{N}$,
there is always at least one family $\mc{M}$ such that
$\mc{N} = \mc{M}'$, namely the family
$\mc{M}=\{N\cup\{n+1\}: N\in \mc{N}\}\cup\{\emptyset\}$. Note that
$\mc{M}$ is also normalized. Now, if $\F$ is an independent family,
then take $\mc{M}$ such that $\mc{M}' = \F^*$, and
$\mc{T} = J(\mc{M})^*$, which is an independent family as seen in the
proof of Proposition~\ref{tudo*}. Then $\mc{T}^* = \mc{M}$, by
Proposition~\ref{tudo+}, and thus
$\mc{T}_\downarrow = J((\mc{T}^*)')^* = J(M')^*= J(\F^*)^* = \F$, by
Lemma~\ref{techlemma}. This shows that, given any independent family
$\F$ there is always a family $\mc{T}$ such that
$\F=\mc{T}_\downarrow$.  We will refer to this family $\mc{T}$ as
\emph{the trivial parent} of $\F$. In other words, the operator
$\square'$ is surjective on normalized families, while the operator
$\square_\downarrow$ is surjective on independent families. It is
clear that if an independent family satisfies the Frankl conjecture,
so does its trivial parent.

Poonen conjectured in \cite{Poo92} that every union-closed family is
either a power set or has an element in strictly more than half the
sets.

\begin{conjecture}
  \label{poonen}
  Let $\mc{F}$ be a union-closed family of sets. Unless $\F$ is a
  power set, it contains an element that appears in strictly more than
  half of the sets
\end{conjecture}

We propose the following seemingly weaker version of Poonen
conjecture.
\begin{conjecture}
  \label{nossaguess}
  Let $\mc{F}$ be a union-closed family such that the most frequent
  element belongs to exactly half the sets in $\F$. Then $\F$ must be
  a power set.
\end{conjecture}

We can now prove that these two conjectures are in fact
equivalent. The advantage of this new statement is that it only
concerns families \emph{sharply} satisfying Frankl's conjecture,
apparently making no claim on the original conjecture: naturally, the
non-trivial part is proving that Conjecture~\ref{nossaguess} implies
Frankl's conjecture.

\begin{theorem}\label{poonenequiv}
  Conjectures~\ref{poonen} and~\ref{nossaguess} are equivalent.
\end{theorem}

\begin{proof}
  Clearly Conjecture~\ref{poonen} implies Conjecture~\ref{nossaguess}.
  We will prove that Conjecture~\ref{nossaguess} implies
  Conjecture~\ref{origconjecture}. This is enough since
  Conjecture~\ref{nossaguess} together with
  Conjecture~\ref{origconjecture} implies Conjecture~\ref{poonen}.
  Assume that Conjecture~\ref{nossaguess} holds while Frankl's
  conjecture does not. By Theorem~\ref{franklsalz} there would be a
  family $\F$ such that $\F^*$ is an $n$-normalized counterexample to
  the Salzborn formulation of the conjecture. Put
  $\lambda=\max\{|I|: I\in J(\F)\}$, and choose a set $I\in J(\F)$
  such that $|I|=\lambda$. Obviously, $|I|<\frac{n+1}{2}$. If we
  consider $n-2\lambda+1$ successive trivial parents of $\F$ we obtain
  a family $\mc{T}$ such that $T=I\cup\{n+1,\ldots , 2n-2\lambda +1\}$
  is a maximal element of $J(\mc{T})$ and
  $\frac{|T|}{|\mc{T}|}=\frac{n-\lambda +1}{2n-2\lambda +2}=\frac
  12$. By Proposition~\ref{tudo+}, $\mc{T}=J(\mc{T})^{**}$. But, as
  seen in Proposition~\ref{tudo*}, $J(\mc{T})^*$ is an independent
  family, and therefore, by Lemma~\ref{techlemma} applied to
  $\mc{F}=\mc{G}=J(\mc{T})^*$, we have that
  $T\in J(\mc{T})=J(J(\mc{T}^*)^*)=(J(\mc{T})^*)^\iota$.  Hence, the
  largest set $(J(\mc{T})^*)^{\iota_j}$, for $j\in U(J(\mc{T})^*)$, is
  $T$ with cardinal $n-\lambda +1$. Hence, the most frequent element
  in $J(\mc{T})^*$ has frequency $n-\lambda +1$. Using
  Proposition~\ref{dualnormalized}, we see that
  $|J(\mc{T})^*|=|(J(\mc{T})^*)^*|=|\mc{T}|=2n-2\lambda +2.$ By the
  hypothesis, we deduce that $J(\mc{T})^*$ is a power set, but that is
  absurd since, by construction, $\mc{T} = (J(\mc{T})^*)^*$ has a
  singleton.
\end{proof}

We now prove the Frankl conjecture for families that are descendents
of power set families. To do so, we start with a technical lemma.

\begin{lemma} \label{ineq binomial} Let $n\geq 6$ be an integer and
  $2\leq k \leq \frac n 2$. Then
  $$2^{k+1}-1\leq \sum_{s=0}^{k-1} {n\choose s}.$$ 
\end{lemma}

\begin{proof}
  For $k=2$ we have that $2^{k+1}-1=7$ and
  $\sum_{i=0}^{k-1} {n\choose i}=n+1\geq 7.$ For $k=3$, we have that
  $2^{k+1}-1=15$ and $\sum_{s=0}^{k-1} {n\choose s}=\frac{n^2+n+2}2$,
  which is greater than $15$ if $n\geq 6$.

  Now assume that $k\geq 4$. Since the sum of each row of Pascal's
  triangle is half of the sum of the next row, and the triangle is
  symmetrical, we have that
  $$2^{k+1}-1\leq 2^{k+1}=\sum_{s=0}^{k+1}{{k+1}\choose s}\leq
  \sum_{s=0}^{\left\lfloor\frac{k+2}2\right\rfloor}{{k+2}\choose
    s}\leq \sum_{s=0}^{k-1}{{k+2}\choose s}\leq \sum_{s=0}^{k-1}
  {n\choose s}.$$
\end{proof}

\begin{theorem}\label{descpower}
  If a family $\mc{F}$ is a descendent of a power set, then it
  satisfies the Frankl Conjecture.
\end{theorem}

\begin{proof}
  We may assume that $n\geq 6$ by \cite[Theorem 10]{LoF94}.

  Suppose $\mc{F}=\mc{P}([n])$, for some $n$, and build its
  correspondent normalized family $\mc{N}=\F^*$. Then we have
  $|\mc{N}|=2^n$, $|U(\mc{N})|=2^n-1$, $|J(\mc{N})|=n$, and
  $|I|=2^{n-1}$ for every $I\in J(\mc{N})$.  In fact, $\mc{N}$ has
  $n\choose k$ distinct sets of size $2^n-2^k$, for every
  $k=0,1,\ldots,n$, which correspond to the sets containing the
  indices of the sets in $\F_F$, for $F\in \F$ with $|F|=k$. Indeed,
  it is easy to see that $\F_F\neq \F_G$ whenever $F\neq G$, and that
  for every $F\in\F$ with $|F|=k$ one has
  $$|\F_F|=2^{n-k}\sum\limits_{i=1}^k {k\choose
    i}=2^{n-k}(2^k-1)=2^n-2^{n-k}.$$

  We will prove that the normalized families obtained in successive
  reductions satisfy the Salzborn formulation of the conjecture. This
  suffices in view of Propositions~\ref{tudo+} and~\ref{tudo*}, and
  Theorem~\ref{franklsalz}.

  Consider the subfamilies $\mc{N}_k\subseteq \mc{N}$ defined by
  $\mc{N}_k=\{N\in \mc{N}: |N|=2^n-2^{n-k}\}$. With this notation, we
  have that $J(\mc{N})= \mc{N}_1$. Let $\mc{N}_k^{(i)}$ the subfamily
  obtained by the sets descending from the sets in $\mc{N}_k$ after
  $i$ $\square'$ reductions. The sets in $\mc{N}_k^{(i)}$ have
  $2^n-2^{n-k}-i$ elements, and, clearly, while doing the successive
  reductions, when the last set descending from a set in $\mc{N}_k$ is
  removed, then all sets descending from a set in $\mc{N}_{k+1}$ are
  irreducible, as they are then minimal.

  We apply the reduction process by removing all sets descending from
  $\mc N_1$, then all sets descending from $\mc N_2$, and so on.

  Put $J(\mc{N})=\{\emptyset,I_1,\ldots, I_n\}$, where $I_i$ is a set
  containing the indexes of sets in $\F$ having the element
  $i\in U(\F)$. Assume, without loss of generality, that the removed
  minimal set is $I_1$.  When we do so, $n-1$ sets of $\mc{N}_2'$,
  with size $2^n-2^{n-2}-1$, become irreducible, namely the sets
  coming from the sets $I_1\cup I_i$, for $2\leq i \leq n$. This
  happens because $\{1,i\}\in \F$, which implies that the only
  irreducible sets of $\mc{N}$ containing the index of $\{1,i\}$ are
  $I_i$ and $I_1$, and hence $I_1\cup I_i$ cannot be written as union
  of sets not involving $I_1$.  At this instance, we thus have
  irreducible sets of size $2^n-2^{n-2}-1$. When we do the $\ell$-th
  reduction, we must have irreducible sets with at least
  $2^n-2^{n-2}-\ell$ elements, because every set that is then
  irreducible belongs to $\mc{N}_k^{(\ell)}$ for some $k\geq 2$. Since
  $$2^n-2^{n-2}-\ell \geq \frac{2^n-\ell}2 \iff \ell\leq 2^{n-1},$$ we have the
  conjecture verified up until $\mc{N}^{(2^{n-1})}$.

  After doing $2^{n-1}$ reductions, our smallest irreducible elements
  are sets descending from $\mc N_{\frac{n+1}{2}}$, if $n$ is odd
  (since the $2^{n-1}$ removed sets are precisely all sets from
  $\mc N_1\cup \cdots \cup \mc N_{\frac{n-1}{2}}$), or sets descending
  from $\mc N_{\frac{n}{2}}$, if $n$ is even (since the $2^{n-1}$
  removed sets are precisely all sets from
  $\mc N_1\cup \cdots \cup \mc N_{\frac{n}{2}-1}$ together with half
  the sets from $\mc N_{\frac n 2}$). Hence, it suffices to prove
  that, when removing sets coming from $\mc N_k$ with
  $k\geq \frac n 2$, we always have irreducible sets of size greater
  than half the respective universe.

  Let $k\geq \frac n 2$. When we remove the $i$-th set coming from
  $\mc N_k$ (this corresponds to the
  $\left(\sum_{r=1}^{k-1}{n \choose r}+i\right)$-th reduction in
  total), the size of the remaining elements, if any, coming from
  $\mc N_k$ is $2^n-2^{n-k}-\sum_{r=1}^{k-1}{n \choose r}-i$. If there
  are no remaining sets coming from $\mc N_k$, then the sets coming
  from $\mc N_{k+1}$ are now irreducible, and larger. The total
  number of sets in $\mc N$ is $2^n-\sum_{r=1}^{k-1}{n \choose
    r}-i$. We have that
  \begin{align*}
    2^n-2^{n-k}-\sum\limits_{r=1}^{k-1}{n \choose r}-i\geq
    \frac{2^n-\sum\limits_{r=1}^{k-1}{n \choose r}-i}{2}
    \iff & 2^{n-1}-2^{n-k}-\frac{\sum\limits_{r=1}^{k-1}{n \choose
           r}}{2}\geq \frac{i}{2}\\ 
    \iff &i\leq 2^{n}-2^{n-k+1}-{\sum\limits_{r=1}^{k-1}{n \choose
           r}}. \numberthis \label{eqn} 
  \end{align*}
  We know that $i\leq {n \choose k}$, since that is total number of
  sets in $\mc N_k$, so it suffices to prove that
  $$ {n \choose k}\leq 2^{n}-2^{n-k+1}-\sum\limits_{r=1}^{k-1}{n \choose r}.$$
  We have that
  \begin{align*}
    {n \choose k}\leq 2^{n}-2^{n-k+1}-\sum\limits_{r=1}^{k-1}{n
    \choose r}
    \iff &2^{n-k+1}\leq 2^{n}-\sum\limits_{r=1}^{k}{n \choose r}\\
    \iff &2^{n-k+1}-1\leq 2^{n}-\sum\limits_{r=0}^{k}{n \choose r}\\
    \iff  &2^{n-k+1}-1\leq \sum\limits_{r=k+1}^{n}{n \choose r}\\
    \iff  &2^{n-k+1}-1\leq \sum\limits_{r=0}^{n-k-1}{n \choose r}.
  \end{align*}
  If $k\leq n-2$, this follows from Lemma~\ref{ineq binomial}. The
  only case missing is the case $k=n-1$. Replacing $k$ by $n-1$ in
  (\ref{eqn}), we get that $i\leq n-2$, so we can remove the first
  $n-2$ sets of $\mc N_{n-1}$. Since $\mc N_{n-1}$ has $n$ sets, when
  one of the two last sets is removed, then the universe (of the set
  in $\mc N_n$) becomes irreducible and the conjecture is satisfied.
\end{proof}
\section{Future work}

There are some things that we believe might be interesting to do using
the construction of families using the reductions introduced in this
paper.  The first big question would be classifying the families which
descend from power sets. It would also be interesting to uncover some
relations between relatives. For example, could it be proved that if a
descendent from a certain family satisfies the Frankl Conjecture, then
all its siblings do?  If all children from a certain parent satisfy
the conjecture, then so does the parent?  If one parent satisfies the
conjecture, then does every parent satisfy it too? If every parent of
a family satisfies the conjecture, then so does that family?

Also, beyond the results presented in this paper, is it possible to
further reduce the space of families for which it suffices to prove
the conjecture? In particular, can it be reduced the $T_1$ families?
Or to some class of supratopological spaces smaller than $\T{D}$ or
$\T{iD}$ spaces, e.g.~$\T{DD}$ spaces?  Since we have proven that the
conjecture holds for $\T{FF}$ spaces, and it is enough to prove it for
$\T{D}$ spaces, what can it be said about $\T{F}$ spaces?
\section*{Acknowledgements}

The authors were partially supported by CMUP, member of LASI, which is
financed by national funds through the FCT --- Fundação para a Ciência
e a Tecnologia, I.P., under the projects with reference
UIDB/00144/2020 and UIDP/00144/2020.
\bibliographystyle{plain}
\bibliography{SupraTopNormalizedFranklConj}

\begin{thebibliography}{10}

\bibitem{AN98}
T.~Abe and B.~Nakano.
\newblock Frankl's conjecture is true for modular lattices.
\newblock {\em Graphs and Combinatorics}, 14:305--311, 1998.

\bibitem{AullThron}
C.~E. Aull and W.~J. Thron.
\newblock Separation axioms between {$T_0$} and {$T_1$}.
\newblock {\em Indagationes Mathematicae (Proceedings)}, 65:26--37, 1962.

\bibitem{BBE13}
I.~Balla, B.~Bollob{\'a}s, and T.~Eccles.
\newblock Union-closed families of sets.
\newblock {\em Journal of Combinatorial Theory, Series A}, 120:531--544, 2013.

\bibitem{BCST15}
H.~Bruhn, P.~Charbit, O.~Schaudt, and J.~A. Telle.
\newblock The graph formulation of the union-closed sets conjecture.
\newblock {\em European Journal of Combinatorics}, 43:210--219, 2015.

\bibitem{BS13}
H.~Bruhn and O.~Schaudt.
\newblock The journey of the union-closed sets conjecture.
\newblock {\em Graphs and Combinatorics}, 31:2043--2074, 2015.

\bibitem{Cam22}
S.~Cambie.
\newblock Better bounds for the union-closed sets conjecture using the entropy
  approach.
\newblock Preprint, ar{X}iv:2212.12500, 2022.

\bibitem{Car16}
A.~Carvalho.
\newblock Frankl {C}onjecture.
\newblock Master's thesis, University of Porto, 2016.

\bibitem{Csaszar02}
\'A. Cs\'asz\'ar.
\newblock Generalized topology, generalized continuity.
\newblock {\em Acta Mathematica Hungarica}, 96(4):351--357, 2002.

\bibitem{Duffus}
Dwight Duffus.
\newblock Open problem session.
\newblock In Ivan Rival, editor, {\em Graphs and Order: The Role of Graphs in
  the Theory of Ordered Sets and Its Applications}, volume 147 of {\em {NATO
  ASI series. Series C: Mathematical and Physical Sciences}}, page 525. D.
  Reidel Publishing Company, 1985.

\bibitem{LF94}
G.~Lo Faro.
\newblock A note on the union-closed sets conjecture.
\newblock {\em Journal of the Australian Mathematical Society, Series A},
  57:230--236, 1994.

\bibitem{LoF94}
G.~Lo Faro.
\newblock Union-closed sets conjecture: improved bounds.
\newblock {\em Journal of Combinatorial Mathematics and Combinatorial
  Computing}, 16:97--102, 1994.

\bibitem{Gil22}
J.~Gilmer.
\newblock A constant lower bound for the union-closed sets conjecture.
\newblock Preprint, ar{X}iv:2211.09055, 2022.

\bibitem{YJ09}
Y.~Jiang.
\newblock A generalization of the union-closed set conjecture.
\newblock \url{http://math.stanford.edu/~jyj/Union_closed_set_conj.pdf}.
\newblock [online: accessed 2009].

\bibitem{JTPV97}
R.~T. Johnson and T.~P. Vaughan.
\newblock On union-closed families, {I}.
\newblock {\em Journal of Combinatorial Theory, Series A}, 84(242-249), 1998.

\bibitem{karpas2017resultsunionclosedfamilies}
Ilan Karpas.
\newblock Two results on union-closed families.
\newblock Preprint, arXiv:1708.01434v1, 2017.

\bibitem{Mashhour}
A.~S. Mashhour, A.~A. Allam, F.~S. Mahmoud, and F.~H. Khedr.
\newblock On supratopological spaces.
\newblock {\em Indian Journal of Pure and Applied Mathematics}, 14(4):502--510,
  1983.

\bibitem{Mor06}
R.~Morris.
\newblock {FC}-families, and improved bounds for {F}rankl's conjecture.
\newblock {\em European Journal of Combinatorics}, 27:269--282, 2006.

\bibitem{PankajamSivaraj}
V.~Pankajam and D.~Sivaraj.
\newblock Some separation axioms in generalized topological spaces.
\newblock {\em Boletim da Sociedade Paranaense de Matemática}, 31(1):29--42,
  2013.

\bibitem{Poo92}
B.~Poonen.
\newblock Union-closed families.
\newblock {\em Journal of Combinatorial Theory, Series A}, 59:253--269, 1992.

\bibitem{Rei03}
D.~Reimer.
\newblock An average set size theorem.
\newblock {\em {Combinatorics, Probability and Computing}}, 12:89--93, 2003.

\bibitem{Rei00}
J.~Reinhold.
\newblock Frankl's conjecture is true for lower semimodular lattices.
\newblock {\em Graphs and Combinatorics}, 16:115--116, 2000.

\bibitem{Rod12}
E.~Rodaro.
\newblock Union-closed vs upward-closed families of finite sets.
\newblock Preprint, ar{X}iv:1208.5371v2, 2012.

\bibitem{SR89}
D.G. Sarvate and J-C. Renaud.
\newblock Improved bounds for the union-closed sets conjecture.
\newblock {\em Ars Combinatoria}, 29:181--185, 1990.

\bibitem{Vau02}
T.~P. Vaughan.
\newblock Families implying the {F}rankl conjecture.
\newblock {\em European Journal of Combinatorics}, 23:851--860, 2002.

\bibitem{Woj92}
P.~W{\'o}jcik.
\newblock Union-closed families of sets.
\newblock {\em Discrete Mathematics}, 199:173--182, 1999.

\bibitem{Yu23}
L.~Yu.
\newblock Dimension-free bounds for the union-closed sets conjecture.
\newblock {\em Entropy}, 25(5):767, 2023.

\end{thebibliography}
\end{document}